\documentclass[11pt]{amsart}
\usepackage{amsmath,amsfonts,amsthm,amssymb,amscd,url,mathrsfs, graphicx}
\usepackage{tikz, amsmath,amssymb,amsthm,stmaryrd,graphics}
\usetikzlibrary{shapes,arrows}

\newcommand{\Z}{\mathbb{Z}}

\DeclareFontFamily{OT1}{rsfs}{}
\DeclareFontShape{OT1}{rsfs}{n}{it}{<-> rsfs10}{}
\DeclareMathAlphabet{\mathscr}{OT1}{rsfs}{n}{it}

\DeclareMathOperator{\Pbold}{\mathbf{P}}
\DeclareMathOperator{\Lbold}{\mathbf{L}}
\DeclareMathOperator{\Pnought}{\mathbf{P}_0}
\DeclareMathOperator{\Pone}{\mathbf{P}_1}
\DeclareMathOperator{\Ptwo}{\mathbf{P}_2}
\DeclareMathOperator{\Pthree}{\mathbf{P}_3}
\DeclareMathOperator{\Lnought}{\mathbf{L}_0}
\DeclareMathOperator{\Lone}{\mathbf{L}_1}
\DeclareMathOperator{\Ltwo}{\mathbf{L}_2}

\DeclareMathOperator{\spaceS}{\mathscr{S}}
\DeclareMathOperator{\spaceP}{\mathscr{P}}
\DeclareMathOperator{\spaceL}{\mathscr{L}}
\DeclareMathOperator{\spaceI}{\mathscr{I}}

\newtheorem{prop}{Proposition}[section]
\newtheorem{thm}[prop]{Theorem}

\newtheorem{cor}[prop]{Corollary}
\newtheorem{lem}[prop]{Lemma}

\newtheorem*{esconj}{ Erd\"os-Szemer\'edi conjecture}
\newtheorem*{stthm}{Szemer{\'e}di-Trotter theorem}
\newtheorem*{ffspthm}{Sum-product theorem for prime fields}

\newtheorem*{grd}{Geometric Ruzsa triangle inequality}

\newtheorem*{defn*}{Definition}

\numberwithin{equation}{section}
\title{On growth in an abstract plane }
\author[N.~Gill]{Nick Gill}
\author[H.~A.~Helfgott]{Harald A.~Helfgott}
\author[M.~Rudnev]{Misha Rudnev}

\address{Nick Gill\newline
Department of Mathematics\newline
The Open University\newline
Walton Hall, Milton Keynes, MK7 6AA, United Kingdom}
\email{n.gill@open.ac.uk}

\address{Harald A.~Helfgott\newline
D\'epartement de math\'ematiques et applications\newline
\'Ecole normale sup\'erieure\newline
45 rue d'Ulm, F-75230 Paris,France}
\email{helfgott@dma.ens.fr}

\address{Misha Rudnev\newline
School of Mathematics\newline
University Walk\newline
Bristol, BS8 1TW\newline
United Kingdom}
\email{m.rudnev@bristol.ac.uk}

\begin{document}
\begin{abstract}
There is a parallelism between growth in arithmetic combinatorics and
growth in a geometric context. While, over $\mathbb{R}$ or $\mathbb{C}$,
geometric statements on growth often have geometric proofs, what little 
is known over finite fields rests on arithmetic proofs.
We discuss strategies
for geometric proofs of growth over finite fields, and show that growth
can be defined and proven in an abstract projective plane -- even one with
weak axioms.
\end{abstract}
\maketitle

\section{Introduction}

The ties between arithmetic combinatorics and combinatorial geometry have
been close and fruitful. The underlying correspondence connects problems
involving addition and multiplication, on the one hand, with geometric
problems on incidence, on the other hand. There is more than one way to state
and apply this correspondence; the best way depends on the problem,
and finding it takes some skill.


\subsection{Real numbers}

In the Euclidean plane, this paradigm is illustrated by Elekes' work on the Erd\"os--Szemer\'edi conjecture \cite{ES}:

\begin{esconj}
For every $\varepsilon \in (0,1)$ there exists $C_\varepsilon>0$ such that, for all finite sets $A\subset \mathbb R$,
\begin{equation*}\label{esc}
|A+A|+|A\cdot A|\geq C_\varepsilon|A|^{2-\varepsilon}.
\end{equation*}
\end{esconj}
Note that we are using the following definitions:
\[\begin{aligned}
A+A &:= \{a_1+a_2:\,a_{1,2}\in A\} , \\
A\cdot A &:= \{a_1\cdot a_2:\,a_{1,2}\in A\},\\
|A| &:= \text{the number of elements of $A$.}\end{aligned}\]

This conjecture should be thought of as a statement concerning the arithmetical {\it growth} of sets in $\mathbb{R}$: it proposes that any set must grow quickly under the operations of multiplication and addition.

Elekes' proved the conjecture for $\varepsilon \in [\frac34, 1)$. 
(The conjecture was known before for $\varepsilon \in [\frac{14}{15}, 1)$;
see \cite{F, Na, ES}.) 
Elekes used the following geometrical result \cite{ST}:

\begin{stthm}
Let $P$ be a finite set of points and $L$ a finite set of lines in the real plane $\mathbb{R}^2$. Then
$$I(P,L) \leq 4|P|^{\frac23}|L|^{\frac23}+4|P|+|L|.$$
\end{stthm}

Note that $I(P,L)$ is the number of incidences between $P$ and $L$:
$$I(P,L):= |\{(p,l)\in P\times L \, \mid \, p\in l\}|.$$
Elekes' proof takes a single paragraph \cite{E}: He considers the number of incidences between the set of points
$$P:=\{(a,b)\, \mid \, a\in A+A, b\in A\cdot A\}$$
and the set of lines
$$L:=\{(x,y) \, \mid \, y=a(x-b)\}$$
and the result falls out immediately.

Subsequent improvements to Elekes' work (the best is due to Solymosi \cite{S}, whose approach also generalises to complex numbers \cite{KR}) also make use of geometric properties of the reals including, in particular, the fact that they are an ordered field.

This connection between the geometry and arithmetic of $\mathbb{R}$ has been pushed further. For instance, the three-dimensional point-line incidence theorem of Guth-Katz \cite{GK} has been applied to establish the following near-optimal sum-product type statement \cite{IRR, RR}: For every $\varepsilon\in (0,1)$ there exists $C_\varepsilon >0$ such that
 $$
 |AA+AA|,\;|(A+A)\cdot (A+A)|\geq C_\varepsilon|A|^{2-\varepsilon}.
 $$

\subsection{Finite fields}
Much less is known about growth over prime fields ${\Z}/{p\Z}$ (and other finite fields). This is due at least in part to the fact
that there is no  ordered geometry to play with over ${\Z}/{p\Z}$.

A major step in studying growth in finite fields was a paper of Bourgain, Katz, and Tao \cite{BKT} in which the following qualitative result was proved:

\begin{ffspthm}
Fix $\delta\in (0,1)$. There exist $\varepsilon \in (0,1)$ and $C>0$ such that for any set $A$ in ${\Z}/{p\Z}$ with $|A|<p^{1-\delta}$ we have \begin{equation*}
\max(|A+A|, |A\cdot A|) \geq C |A|^{1+\varepsilon}.
\end{equation*}
\end{ffspthm}

As a corollary of this result, \cite{BKT} derived a variety of incidence-type results including a qualitative `Szemer\'edi-Trotter theorem for prime fields'. These corollaries demonstrate that the connection between geometry and arithmetic remains strong even over prime fields.  However, the result itself was proved using non-geometrical ideas.

Subsequent work yielding quantitative geometric results has followed a
similar path: geometrical results in the projective plane over ${\Z}/{p\Z}$
have been established by reducing them to algebraic sum-product type
relations. The explicit bounds that have been established this way have
been rather weaker than in the Euclidean case \cite{HR, J}. Without going
into detail, let us point out that \cite{J} proves a state-of-the-art exponent $\frac{1}{662}$ in a version of the Szemer\'edi-Trotter theorem, where in the Euclidean case the exponent is $\frac{1}{6}.$

\subsection{Geometric proofs for finite fields}

The absence of properly `geometric proofs' for results over finite fields is noteworthy. It appears sensible to try to
find more idiomatric proofs within geometry (which may yield stronger explicit bounds).

Of course geometric axioms often imply an algebraic structure; for example, the axiom of Pappus implies a field structure. Still, there is a natural sense in which a proof can be said to
happen within geometry, rather than by reduction to an algebraic argument. This
is so even if some ideas from work on groups or fields are taken. We give a simple example of this type in the appendix to this note, showing how just the little Desargues axiom can take over an
argument underlying Ruzsa's distance inequality, one of the key tools in additive combinatorics \cite{R}.

The litmus test here is whether one can give a combinatorial
argument over a projective geometry such that the
argument has no immediate algebraic analogue. This can be forced by having
projective-plane axioms that are too weak for one of the usual algebraic
structures to exist, and yet still proving a meaningful geometric
result in that projective plane. The aim of this paper is to demonstrate
this by giving exactly such a result.

\subsection{Geometric growth}

Our main result concerns a notion of ``growth'' that has already appeared in the literature in relation to the Euclidean plane. For $P$ a set of
points on the plane, define $L(P)$ to be the set of lines
defined by (that is, incident to) some pair of distinct points of $P$.
 As a corollary of the Szemer\'edi-Trotter theorem, Beck proved the following statement  \cite{B}: There exists an absolute constant $c>0$, such that if $P\subset \mathbb R^2$ is a set of points, with no more than $c|P|$ points being collinear,
then $|L(P)|\geq c|P|^2$. We think of the set $P$ as {\it growing} under the operation of `defining lines'.

To state our main result we need to develop this idea a little. (Note that all relevant definitions are given in \S\ref{s: background}. We also recommend \cite{D, HP, Si} for foundations of projective geometry.)

Let $\spaceS=(\spaceP,\spaceL,\spaceI)$ be an abstract projective plane and let $\mathbf{P}$ be a set of points in $\spaceP$. Now define a sequence of sets as follows:
\begin{itemize}
\item $\Pnought=\Pbold$;
\item $\mathbf{L}_i(\mathbf{P}), i=0,1,2,\dots$ is the set of lines incident with at least two points of $\mathbf{P}_{i-1}(\mathbf{P})$ (say $\mathbf{L}_i(\mathbf{P})$ is the set of points {\it defined by} $\mathbf{P}_{i-1}(\mathbf{P})$);
\item $\mathbf{P}_i(\mathbf{P}), i=1,2,\dots$ is the set of points incident with at least two lines of $\mathbf{L}_{i-1}(\mathbf{P})$ (say $\mathbf{P}_i(\mathbf{P})$ is the set of points {\it defined by} $\mathbf{L}_{i-1}(\mathbf{P})$).
\end{itemize}

Where there is no danger of ambiguity we will write $\mathbf{P}_i$ and $\mathbf{L}_i$ rather than $\mathbf{P}_i(\mathbf{P})$ and $\mathbf{L}_i(\mathbf{P})$. Our primary result is the following:

\begin{thm}\label{t: main}
Let $\mathbf{P}$ be a finite set of points in an abstract projective plane. Then one of the following statements holds:
\begin{enumerate}
 \item $|\Pthree|\geq \frac14 |\mathbf{P}|^2$;
\item $\Pone$ is equal to the set of points of a projective subplane, or to the set of points of a projective subplane minus one;
\item $\mathbf{P}$ is a degenerate subplane.
\end{enumerate}
\end{thm}

Theorem~\ref{t: main} asserts that, provided the starting set isn't close to filling a (possibly degenerate) subplane, then the sequence of sets described above {\it grows} in size. Note that no assumption has been made with regard to the axioms of Desargues, or Pappus: our result holds for planes that cannot be coordinatized by skew-fields.

A question remains: can this line of study be extended to give a geometric
proof of Szemer\'edi-Trotter result over $\mathbb{P}^2(\mathbb{Z}/p\mathbb{Z})$, or over
arbitrary finite projective planes?

\subsection{Acknowledgments}

Nick Gill would like to thank the University of Bristol, to which he has been a frequent visitor during the writing of this paper. Harald Helfgott thanks MSRI
(Berkeley) for its support during a stay there.

\section{Incidence Systems}\label{s: background}

An {\it incidence system} $\spaceS$ is a triple $(\spaceP,\spaceL, \spaceI)$ where $\spaceP$ and $\spaceL$ are sets and $\spaceI\subseteq \spaceP\times \spaceL$. We refer to $\spaceP$ as the set of {\it points},  $\spaceL$ as the set of {\it lines} and $\spaceI$ as the set of {\it incidences} between points and lines. Thus, if $(\wp, \ell)\in \spaceI$, then we say the the point $\wp$ {\it is incident with} the line $\ell$; we will sometimes abuse language by saying things like $\wp$ ``lies on'' $\ell$ or $\ell$ ``contains'' $\wp$, etc.

An incidence system $(\spaceP,\spaceL, \spaceI)$ is called {\it finite} if the sets $\spaceP$ and $\spaceL$ are finite. We define the set of {\it dual incidences} to be
$$\spaceI^D=\{(\ell, \wp)\in \spaceL\times \spaceP \, \mid \, (\wp,\ell)\in \spaceI\}.$$
Then the {\it dual} of $\spaceS$ is the incidence system $(\spaceL, \spaceP, \spaceI^D)$.

\subsection{Projective planes}
A {\it projective plane} $\spaceS$ is an incidence system $(\spaceP, \spaceL, \spaceI)$ satisfying the following axioms:
\begin{itemize}
\item[(P1)] Any two distinct points are incident with exactly one line;
\item[(P2)] Any two distinct lines are incident with exactly one point;
\item[(P3)] There exists a {\it quadrilateral}, i.e. a set of four points, no three of which are incident with the same line.
\end{itemize}
Observe that the dual of a projective plane is also a projective plane.

Property (P2) allows us to abuse language a little more: we say that two lines $\ell_1, \ell_2$ in a projective plane {\it intersect} at a point $\wp$, meaning that $\wp$ is the unique point incident with both $\ell_1$ and $\ell_2$; in this instance we write $\wp=\ell_1\cap \ell_2$.

The standard example of a projective plane is $PG(2,K)$ where $K$ is any skew field. This is constructed as follows: let $V$ be a $3$-dimensional vector space over $K$; define $\spaceP$ to be the set of $1$-dimensional subspaces of $V$, $\spaceL$ to be the set of $2$-dimensional subspaces and define
$$\spaceI =\{(\wp,\ell)\in \spaceP\times \spaceL \, \mid \, \wp\subset \ell\}.$$
Now set $PG(2,K)=(\spaceP, \spaceL, \spaceI)$; it is an easy matter to check that $PG(2,K)$ is a projective plane.

When $K$ is a finite field of order $q$, the projective plane $PG(2,K)$ is known as {\it the Desarguesian plane of order $q$}. This is because these planes are the only finite projective planes which satisfy the configuration of Desargues. (This is a result of Hilbert \cite[p.28]{D}; the configuration of Desargues is defined in the appendix to this paper.)

\subsection{Other incidence systems}
An incidence system which satisfies (P1) is called a {\it linear space}. The linear space is called {\it regular} if every line is incident with the same number of points, $k$. In the literature a regular linear space $(\spaceP, \spaceL, \spaceI)$ for which $\spaceP$ is finite of order $v$ is also known as a $2-(v,k,1)$ {\it design}.

Let $\mathbf{P}$ be a set in a projective plane $(\spaceP, \spaceL, \spaceI)$, and consider the sets
$$\Pnought, \Pone, \Ptwo, \dots \textrm{ and } \Lnought, \Lone, \Ltwo, \dots$$ as defined in the introduction. For $i=0,1,2,\dots$, consider the triple $(\mathbf{P}_i, \mathbf{L}_i, \mathbf{I}_i)$ where $\mathbf{I}_i$ is the restriction of $\spaceI$ to the set $\mathbf{P}_i\times \mathbf{L}_i$; by definition this triple is a linear space. The same is true of the incidence system $(\mathbf{L}_i, \mathbf{P}_{i+1}, \mathbf{J}_i)$ where $\mathbf{J}_i$ is the restriction of $\spaceI^D$ to the set $\mathbf{L}_i\times \mathbf{P}_{i+1}$.

An incidence system which satisfies (P1) and (P2) but not (P3) is known as a {\it degenerate projective plane}. The following result is easy.

\begin{lem}\label{l: degenerate}
Let $\spaceS=(\spaceP, \spaceL, \spaceI)$ be a degenerate projective plane. Then one of the following holds:
\begin{enumerate}
 \item There exists a line $\ell\in\spaceL$ that is incident with every point in $\spaceP$.
\item There exists a line $\ell$ such that all points but one (which we call $\wp$) is incident with $\ell$. Furthermore all other lines are incident with precisely two points, one of which is $\wp$.
\end{enumerate}
\end{lem}

We will refer to a degenerate projective plane of the second type as a {\it fan}.

\section{Growth in the projective plane}

In this section we prove Theorem~\ref{t: main}.

\subsection{Preparatory lemmata}

Consider a set of points $\mathbf{P}$ in a projective plane $\spaceS$; we write $\mathbf{L}$ for $\Lnought({\mathbf{P}})$, the set of lines defined by $\Pbold$. Below we give a number of elementary results concerning the sets $\Pbold$ and $\mathbf{L}$. Recall that the dual of a projective plane is also a projective plane, hence the results we give also apply to the sets $\mathbf{L}$ and $\Pone$, say.

We note that, provided there does not exist a line incident with all elements of $\Pbold$, we have $\Pbold\subseteq \Pone \subseteq \Ptwo \subseteq \cdots$ and (by duality) $\mathbf{L}\subseteq \Lone \subseteq \Ltwo \subseteq \cdots$. In particular, the two sequences of sets {\it grow} in size; indeed, as we shall see in Corollary~\ref{c: frodo}, growth occurs at every step of the construction.

\begin{lem}\label{lem:parnas}
Let $\ell_1$ and $\ell_2$ be two distinct lines in the plane.
Suppose there are $m_1$ points of $\mathbf{P}$
on $\ell_1$ and $m_2$ points of $\mathbf{P}$ on $\ell_2$. Then there
are $\geq (m_1-1) (m_2-1)$ lines in $\mathbf{L}$.

Moreover: if $\ell_1\cap\ell_2$ does not  lie in $\mathbf{P}$, then $|\mathbf{L}|\geq m_1 m_2$; otherwise,
$|\mathbf{L}|\geq (m_1-1) (m_2-1) + a$, where $a$ is the number of lines of $\mathbf{L}$ going through $P$.
\end{lem}
\begin{proof}
Write $\wp$ for $\ell_1\cap\ell_2$.
Since two distinct lines cannot intersect at more than one point,
every pair of points $(\wp_1,\wp_2)$, $\wp_j\in \ell_j$, $\wp_j\ne \wp$,
determines a different line.
There are $\geq (m_1-1) (m_2 - 1)$ lines thus determined, and they
all are in $\mathbf{L}$, by the definition of $\mathbf{L}$. None of them
goes through $\wp$.
\end{proof}

\begin{lem}\label{lem:frodo}
Let $\mathbf{P}$ be a set of points not all on a line. Let
$\mathbf{L}$ be, as usual, the set of lines they define. Then
$|\mathbf{L}|\geq |\mathbf{P}|$.
\end{lem}
This is the well-known {\it Fisher's inequality} for linear spaces. See \cite{BE} for proof. \qed

\medskip

Lemma~\ref{lem:frodo} implies that the sets we defined in the introduction grow at every step.

\begin{cor}\label{c: frodo}
Let $\Pbold$ be a set of points not all on a line. Then
$$|\Pnought|\leq |\Lnought|\leq |\Pone|\leq |\Lone| \leq |\Ptwo| \leq |\Ltwo|\leq \cdots.$$
\end{cor}
\begin{proof}
We have observed already that, for $i=0, 1, 2\dots$, both the pair $(\Pbold_i, \mathbf{L}_i)$ and the pair $(\mathbf{L}_i, \Pbold_{i+1})$ are linear spaces. Now the result follows from Lemma~\ref{lem:frodo}.
\end{proof}

\subsection{Some results on finite linear spaces}

In this section we consider a finite linear space $\spaceS=(\spaceP, \spaceL, \spaceI)$. We write $v=|\spaceP|$ and $b=|\spaceL|$. We also define $k$ to be the average number of points incident with a line, and $r$ to be the average number of lines incident with a point.

\begin{lem}
If $b=v$ then $\spaceS$ is either a regular linear space or a fan.
\end{lem}

\begin{proof}
Let $c$ be the maximum number of points on a line of $\spaceS$. Assume that there is some line in $\spaceS$ which is incident with $<c$ points; we wish to conclude that $\spaceS$ is a fan.

Define a flag to be a pair $(\wp,\ell)$ where $\wp$ is a point and $\ell$ is a line and $\wp$ is incident with $\ell$. We can count flags in two different ways. The number of flags is equal to
$$f:= bk = vr.$$
Then, by assumption, $f<bc = vc$. This means that $r<c$ and so there is a point, $\alpha$, which is incident with $r_\alpha\leq c-1$ lines. Let $d$ be the minimum number of lines that go through any particular point. Observe that $d\leq r < c$.

Now let $c_1\leq c$ be equal to the number of points on the second-most populous line of $\spaceS$. The total number of points equals $v\leq c+(d-1)(c_1-1)$. On the other hand the number of lines connecting points on the two most-populous lines is less than $v$ and more than $(c-1)(c_1-1) + d$. This means that
\begin{eqnarray*}
&&(c-1)(c_1-1) + d\leq c+(d-1)(c_1-1 \\
&\implies& (c-d)c_1\leq 2(c-d).
\end{eqnarray*}
Since $c>d$ we conclude that $c_1=2$ (note that, in particular, this means that $c\geq 3$ as otherwise all lines contain two points). Thus the total number of lines in $\spaceS$ is $v=1+c(v-c)$ which implies that $v=c+1$; then $\spaceS$ is a fan.
\end{proof}

\begin{lem}\label{l: poop}
If $b=v$ then $\spaceS$ is a projective plane or a fan.
\end{lem}
\begin{proof}

We assume that $\spaceS$ is not a fan; thus, by the previous lemma, $\spaceS$ is regular and every line of $\spaceS$ is incident with the same number, $k$, of points. A simple counting argument implies that $b = \frac{v(v - 1)}{k(k-1)}$.

Now define $n$ to be the integer $k-1$; then $b=v = n^2+n+1$. This is enough to conclude that $\spaceS$ is a projective plane \cite[p.138]{D}.
\end{proof}

\subsection{Intrinsic growth results}

We start with a finite set of points $\Pbold$ in a projective plane $(\spaceP, \spaceL, \spaceI)$. We write $\Pbold_i$ for $\Pbold_i(\Pbold)$ and $\Lbold_i$ for $\Lbold_i(\Pbold)$.

\begin{prop}\label{p: growth}
Let $\mathbf{P}$ be a finite set of points in a projective plane. Then one of the following statements hold.
\begin{enumerate}
 \item $|\Pthree|\geq \frac14 |\Pbold|^2$.
\item more than $\frac12 |\Pbold|$ points of $\Pbold$ lie on a line.
\item $\Ptwo=\Pone$ or $\Ptwo = \Pone\cup\{\wp\}$ for some point $\wp$ in the plane.
\end{enumerate}
\end{prop}
\begin{proof}
Suppose (2) and (3) are false. Thus there are two points in $\Ptwo\backslash \Pone$. Now consider lines through a point $\wp$ which does not lie in $\Pone$. At most one of these lines can conain more than one point of $\Pbold$ (otherwise $p$ would lie in $\Pone$). What's more, by (2), such a line can contain at most $\frac12 |\Pbold|$ points of $\Pbold$. Thus there are at least $\frac12|\Pbold|+1$ lines through $\wp$ which are incident with a point from $\Pbold$.

Now we have two such points, $\wp_a$ and $\wp_b$ in $\Ptwo\backslash \Pone$. Thus there are at least $\frac12|\Pbold|+1$ lines in $\Ltwo$ through $\wp_a$ (resp. through $\wp_b$). These lines must intersect in at least $\left(\frac{|\Pbold|}2\right)^2$ points in $\Pthree$.
\end{proof}

\begin{lem}\label{l: plus}
If $|\Pone|=|\Pbold|+1$ then $\Pbold$ consists of all the points of a projective plane minus one.
\end{lem}
\begin{proof}
Observe that both $(\Pbold, \mathbf{L})$ and $(\mathbf{L}, \Pone)$ are linear spaces. Fisher's inequality and the fact that $|\Pone|=|\Pbold|+1$ implies that, for one of these two linear spaces, $b=v$. Hence one of these two linear spaces is a projective plane or a fan.

If $(\Pbold, \mathbf{L})$ is a projective plane or a fan then $\Pone = \Pbold$ which is a contradiction. Suppose $(\mathbf{L}, \Pone)$ is a fan. The dual of a fan is a fan, so $(\Pone, \mathbf{L})$ is also a fan. But then $\Pbold$ must be the fan minus one point; this is either a fan or a line. In both cases $\Pbold=\Pone$ which is a contradiction. The result follows.
\end{proof}

By Lemma~\ref{l: degenerate}, the following corollary to Proposition~\ref{p: growth} is equivalent to Theorem \ref{t: main}.

\begin{cor}\label{c: growth}
Let $\Pbold$ be a finite set of points in a projective plane. Then one of the following statements hold.
\begin{enumerate}
 \item $|\Pthree(\Pbold)|\geq \frac14 |\Pbold|^2$.
\item $\Pone$ is equal to the set of points of a projective subplane, or to the set of points of a projective subplane minus one.
\item $\Pbold$ is equal to the set of points of a fan.
\item there exists a line $\ell$ that is incident with all points of $\Pbold$ .
\end{enumerate}
\end{cor}
\begin{proof}
Apply Proposition \ref{p: growth} to the set $\Pbold$. If (1) holds there, then we are done. If (3) holds there, then there are two cases: first, assume that $\Ptwo=\Pone$; then $(\Pone, \Lone)$ is a finite linear space with $b=v$. Now Lemma~\ref{l: poop} implies that $(\Pone, \Lone)$ is a projective plane (in which case we are done), or else $(\Pone, \Lone)$ is a fan. If $(\Pone, \Lone)$ is a fan, then $\Pbold$ is equal to the set of points of a fan and we are done. The second possibility is that $\Ptwo=\Pone\cup\{\wp\}$ for some point $\wp\not\in\Pone$. Then Lemma~\ref{l: plus} implies that $\Pone$ is equal to the set of points of a projective plane minus one.

We are left with the possibility that (2) holds in Proposition \ref{p: growth}, i.e. there exists $\ell\in\spaceL$ incident with at least $\frac12|\Pbold|$ points of $\Pbold$. If $\ell$ is incident with all points of $\Pbold$, then we are done; if $\ell$ is incident with all points of $\Pbold$ but one, then $\Pbold$ is equal to the set of points of a fan, and we are done. Thus we may assume that $\Pbold$ contains at least two points $\wp_a$ and $\wp_b$ that are not incident with $\ell$. Let $\mathbf{L}_a$ (resp.  $\mathbf{L}_b$) be the set of lines in $\mathbf{L}$ that are incident with $\wp_a$ (resp. $\wp_b$). Observe that, since $\ell$ is incident with at least $\frac12|\Pbold|$ points of $\Pbold$, $|\mathbf{L}_a|\geq \frac12 |\Pbold|$; similarly $|\mathbf{L}_b|\geq \frac12 |\Pbold|$.

Now the lines in $\mathbf{L}_a$ and in $\mathbf{L}_b$ must intersect in at least $\left(\frac{|\Pbold|}{2}\right)^2$ points of $\Pone$. Since $\Pone\subseteq \Pthree$ the result follows.
\end{proof}
\section{Appendix. A geometric Ruzsa-type inequality}

Although the result discussed in this Appendix is not used to prove the main theorem of this note, we have included it to give an example of  an intrinsically geometric statement that can be proved entirely within the sphere of geometry.
It corresponds to the following result, key within arithmetic combinatorics.

\begin{lem}[Ruzsa \cite{R}]\label{lem:gotor}
Let $A,B$ and $C$ be non-empty subsets of an abelian group. Then
$$
|A-C|\leq \frac{|A-B||B-C|}{|B|}
$$
\end{lem}
This is often called ``Ruzsa's triangle inequality''.
Note that we define
$A-B := \{a-b \mid a\in A, b\in B\}$,
and similarly for the other difference sets.
\begin{proof}
It will be enough to construct an injective map
\[\iota: (A-C) \times B \to (A-B) \times (B-C).\]
For each $x\in A-C$, let $f_A(x)$, $f_C(x)$ be elements of
$A$ and $C$ such that $x = f_A(x)- f_C(x)$. Let
\[\iota(x,b) = (f_A(x)-b,b-f_C(x))\]
for $x\in A-C$, $b\in B$.

To show that $\iota$ is injective, it is enough to show how to deduce
what $x$ and $b$ are, given $\iota(x,b)$. Now
\[x = f_A(x) - f_C(x) = (f_A(x)-b) + (b-f_C(x)),\]
so we can certainly deduce $x$ from $\iota(x,b)$. In turn $x$, determines
$f_A(x)$, and this, together with $f_A(x)-b$, gives us $b = f_A(x)-(f_A(x)-b)$.
\end{proof}
As we can see, the proof rests on the fact that $a-b$ and $b-c$ determine
$a-c$.

To prove the geometric version of this result we need a definition. Let $\alpha$ be a point, $\ell$ a line, in a projective plane $\mathscr{P}$. We say that $\mathscr{P}$ is {\it $(\alpha, \ell)$-Desarguesian} if, whenever two triangles are in perspective from $\alpha$ such that two pairs of their tangents meet on $\ell$, then the third pair of tangents meet on $\ell$. Rather than give a rigorous definition of all of the terms just used (which are all standard), we refer the reader to Figure~\ref{f: desargues}.

If $\mathscr{P}$ is $(\alpha, \ell)$-Desarguesian for all incident pairs $(\alpha, \ell)$ , then we say that $\mathscr{P}$ satisfies the {\it Little Desargues Configuration}. If $\mathscr{P}$ is $(\alpha, \ell)$-Desarguesian for all pairs $(\alpha, \ell)$, incident or otherwise, then we say that $\mathscr{P}$ is {\it Desarguesian}.

\tikzstyle{point}=[circle, draw, fill=black!50,
                        inner sep=0pt, minimum width=2pt]
\begin{center}
\begin{figure}
\begin{tikzpicture}
\node [point] at (0,10) (alpha) {} ;
\node at (0,9.8) {$\alpha$};
\draw (alpha) -- (10,10);
\draw (alpha) -- (10, 8.5);
\draw (alpha) -- (10, 7);

\draw (0,6) -- (10,6) ;
\draw (5,10) -- (2,6);
\draw (8,10) -- (2,6);
\draw (5,10) -- (5,6);
\draw (8,10) -- (5,6);
\draw (4.49,9.34) -- (6.52,6);
\draw (6.535,9.02) -- (6.52,6);
\path [fill=gray] (5,10) -- (4.50,9.32) -- (5,8.52) -- (5,10);
\path [fill=gray] (8,10) -- (6.54, 9.02) -- (6.535,8.05) -- (8,10);
\node at (8,5.8) {$\ell$};
\end{tikzpicture}
\begin{caption}{The $(\alpha, \ell)$-Desargues configuration}\label{f: desargues}
\end{caption}
\end{figure}
\end{center}

Consider now the following geometric set-up. Let $l_{A},l_{B},l_{C}$ be the three lines in Figure~\ref{f: desargues} which are incident to the point $\alpha$, with $l$ the bottom line as given.
Let $A$, $B$, $C$ be point sets supported, respectively, on the lines $l_{A},l_{B},l_{C}$; assume $A$, $B$, $C$ are disjoint from $\alpha$ and $\ell$.
Given two distinct points $x$, $y$, write $\overline{xy}$ for the line connecting $x$ to $y$.
For $x$, $y$ distinct and not both on $\ell$, let $\lbrack x,y\rbrack$ be
the intersection of $\ell$ and $\overline{xy}$. Define
\[[X,Y] = \{\lbrack x,y\rbrack : x\in X, y\in Y\}\]
for any two disjoint sets of points $X$, $Y$ such that either $X$ or $Y$ is
also disjoint from $\ell$.

With this notation we have
\begin{grd}
If $\mathscr{P}$ is $(\alpha, \ell)$-Desarguesian, then
\[
|[A,C]|\leq \frac{|[A,B]||[B,C]|}{|B|}.\]
\end{grd}
\begin{proof}
It will be enough to construct an injective map
\[\iota: \lbrack A,C\rbrack \times B \to \lbrack A,B\rbrack \times
\lbrack B,C\rbrack.\]
For each $p\in \lbrack A,C\rbrack$, let $f_A(p)$, $f_C(p)$ be elements of
$A$ and $C$ such that $p$ lies on the line through $f_A(p)$ and $f_C(p)$.
(Such elements exist by the defnition of $\lbrack A,C\rbrack$.) Let
\[\iota(p,b) = (\lbrack f_A(p),b\rbrack,\lbrack b,f_C(p)\rbrack)\]
for $p\in \lbrack A,C\rbrack$, $b\in B$.

To show that $\iota$ is injective, it is enough to show how to deduce
what $p$ and $b$ are, given $\iota(p,b)$. Since $\mathscr{P}$ is
$(\alpha,\ell)$-Desarguesian, $\lbrack a,b\rbrack$ and $\lbrack b,c\rbrack$
determined $\lbrack a,c\rbrack$ for any $a\in A$, $b\in B$, $c\in C$.
In particular, $\lbrack f_A(p),b\rbrack$ and $\lbrack b,f_C(p)\rbrack$
determine $\lbrack f_A(p),f_B(p)\rbrack = p$. In turn, $p$ determines
$f_A(p)$, and this, together with $\lbrack f_A(p), b\rbrack$, determines $b$.
\end{proof}

Of course, we could have obtained some sort of geometric statement from
Lemma \ref{lem:gotor} by coordinatizing the plane $\mathscr{P}$
(over an alternative division ring; see \cite{HP}). The point is that one
can obtain a natural and simple geometric statement with a natural geometric
proof by transfering the {\em ideas} behind the proof of an arithmetic
statement.

\end{document}